\documentclass[12pt,twoside]{amsart}
\usepackage{amssymb}
\usepackage{amscd}
\usepackage{amsmath}
\usepackage{xypic}

\title{On images of weak Fano manifolds} 
\author{Osamu Fujino} 
\author{Yoshinori Gongyo}
\date{2010/4/28, version 3.03}
\address{Department of Mathematics, Faculty of Science, 
Kyoto University, Kyoto 606-8502, Japan}
\email{fujino@math.kyoto-u.ac.jp}
\address{Graduate School of Mathematical Sciences, 
The University of Tokyo, 3-8-1 Komaba, 
Meguro, Tokyo, 153-8914 Japan.}
\email{gongyo@ms.u-tokyo.ac.jp}
\newcommand{\codim}[0]{{\operatorname{codim}}}
\newcommand{\Proj}[0]{{\operatorname{Proj}}}
\newcommand{\Exc}[0]{{\operatorname{Exc}}}
\newcommand{\Supp}[0]{{\operatorname{Supp}}}
\newcommand{\Spec}[0]{{\operatorname{Spec}}}
\newcommand{\ch}[0]{{\operatorname{char}}}
\newtheorem{thm}{Theorem}[section]
\newtheorem{lem}[thm]{Lemma}
\newtheorem{cor}[thm]{Corollary}

\newtheorem{claim}{Claim}
\newtheorem{conj}[thm]{Conjecture}

\theoremstyle{definition}
\newtheorem{ex}[thm]{Example}
\newtheorem{defn}[thm]{Definition}
\newtheorem{rem}[thm]{Remark}
\newtheorem*{ack}{Acknowledgments}       
\newtheorem*{notation}{Notation}
\newtheorem{say}[thm]{}
\begin{document}
\bibliographystyle{amsalpha+}

\maketitle 
\begin{abstract}
We consider a smooth projective morphism 
between smooth complex projective varieties. 
If the source space is a weak Fano (or Fano) manifold, 
then so is the target space. 
Our proof is Hodge theoretic. 
We do not need mod $p$ reduction 
arguments. 
We also discuss related topics and 
questions.  
\end{abstract}

\tableofcontents
\section{Introduction}\label{sec1} 
Let $f:X\to Y$ be a smooth projective 
morphism between smooth projective 
varieties defined over $\mathbb C$. 
The following theorem is one of the main results of this paper. 

\begin{thm}[{cf.~Theorem \ref{02}}]\label{thm1}
If $X$ is a weak Fano manifold, that is, 
$-K_X$ is nef and big, then so is $Y$.  
\end{thm} 

Although we could not find Theorem \ref{thm1} in the literature, 
it seems to be known to some experts. 
Here, we claim the originality of our proof. 
Our proof of Theorem \ref{thm1} is Hodge theoretic. 
We do not need mod $p$ reduction arguments. 
More precisely, we obtain Theorem \ref{thm1} as an application of 
Kawamata's positivity 
theorem (cf.~\cite{kawamata}). By the same method, 
we can recover the well-known result on Fano manifolds. 

\begin{thm}[{cf.~Theorem \ref{03}}]\label{thm2} 
If $X$ is a Fano manifold, that is, $-K_X$ is ample, 
then so is $Y$. 
\end{thm}

Our proof of Theorem \ref{thm2} is completely different from 
the original one by Koll\'ar, Miyaoka, and Mori in \cite{kmm}. 
It is the first proof which does not use mod $p$ reduction 
arguments. 
We raise a conjecture on the semi-ampleness of 
anti-canonical divisors. 

\begin{conj}\label{conj1} 
If $-K_X$ is semi-ample, then so is $-K_Y$. 
\end{conj} 

We reduce Conjecture \ref{conj1} to another conjecture 
on canonical bundle formulas and give affirmative answers 
to Conjecture \ref{conj1}  
in some special cases (cf.~Remark \ref{011} 
and Theorem \ref{torsion}). In this paper, 
we obtain the following theorem, 
which is a key result for the proof of 
Theorem \ref{thm1} and Theorem \ref{thm2}. 

\begin{thm}[{cf.~Theorem \ref{01}}]\label{thm3} 
If $-K_X$ is semi-ample, then $-K_Y$ is nef. 
\end{thm}

We note that the proof of Theorem \ref{thm3} is also an application of 
Kawamata's positivity theorem. We note that 
it is the first time that Theorem \ref{thm3} 
is proved without mod $p$ reduction 
arguments. The reader will recognize that Kawamata's positivity theorem 
is very powerful. 
We can find related topics in 
\cite{zhang} and \cite[Section 3.6]{debarre}. 
Note that both of them depend on mod $p$ reduction arguments. 

We summarize the contents of this paper. 
Section \ref{sec2} is a preliminary section. We recall 
Kawamata's positivity theorem (cf.~Theorem \ref{ka-th}) here. 
In Section \ref{sec3}, we treat log Fano varieties with only kawamata log 
terminal singularities. The result obtained in this section will be used in 
Section \ref{sec4}. 
In Section \ref{sec4}, we prove Theorem \ref{thm1}, 
Theorem \ref{thm2}, and some related theorems. 
In Section \ref{sec5}, we give 
some comments and questions on related topics. 
In the final section:~Section \ref{sec-app}, which is an appendix, 
we give a mod $p$ reduction 
approach to Theorem \ref{thm1}. 

This paper is an expanded version of the second author's 
note \cite{gongyo}. 
It is also the part II of the first author's paper \cite{fujinoa}. 
 
\begin{ack} 
The first author would like to thank Takeshi Abe and 
Kazuhiko Yamaki for fruitful discussions. 
He also thanks Shunsuke Takagi, 
Kazunori Yasutake, and Karl Schwede for useful comments. 
He was partially supported by The Inamori Foundation and by 
the Grant-in-Aid for Young Scientists (A) $\sharp$20684001 from 
JSPS. 
The second 
author would like 
to express his deep gratitude to his supervisor 
Professor Hiromichi Takagi 
for teaching him various techniques of the log minimal model program. 
He also would like to thank Doctor Kazunori Yasutake, 
who introduced the question on weak Fano manifolds 
in the seminar held at the Nihon University in December 2009. 
The second author thanks Doctor Kiwamu Watanabe for reading 
his note \cite{gongyo}, 
Professor Shunsuke Takagi and Doctor Takuzo Okada for fruitful 
discussions. He is indebted to Doctor Daizo Ishikawa. 
Finally, the authors would like to thank Hiroshi Sato for constructing 
an interesting example. 
\end{ack}

We will work over $\mathbb C$, the complex 
number field, from Section \ref{sec2} 
to Section \ref{sec4}.

\section{Preliminaries}\label{sec2} 

We will make use of the standard notation as in the book \cite{km}. 

\begin{notation}
For a $\mathbb Q$-divisor 
$D=\sum _{j=1}^r d_j D_j$ on a normal variety $X$ such that 
$D_j$ is a prime divisor for 
every $j$ and $D_i\ne D_j$ for $i\ne j$, we define  
$$D^{+}=\sum_{d_j> 0}d_j D_j \ \text{\ and\ } \ D^{-}=-\sum_{d_j< 0}d_j D_j.$$ 
We denote the {\em{round-up}} of $D$ by $\ulcorner D\urcorner$. 
Furthermore, let $f:X \to Y$ be a surjective morphism of varieties. 
We define
$$ D^{h}=\sum_{f(D_j)=Y}d_jD_j \ \text{\ and\ } \ D^{v}=D-D^{h}.$$ 

Let $X$ be a normal variety and $\Delta$ an effective $\mathbb Q$-divisor 
on $X$ such that $K_X+\Delta$ is $\mathbb Q$-Cartier. 
Let $\varphi:Y\to X$ be a projective resolution such that the union of 
the exceptional locus of $\varphi$ and the strict transform of $\Delta$ has 
a simple normal crossing support on $Y$. 
We put 
$$
K_Y=\varphi^*(K_X+\Delta)+\sum _i a_i E_i
$$
where $E_i$ is a prime divisor for every $i$ and $E_i\ne E_j$ for 
$i\ne j$. 
The pair $(X, \Delta)$ is called {\em{kawamata log terminal}} 
({\em{klt}}, for short) (resp.~{\em{log canonical}} ({\em{lc}}, for short)) 
pair if $a_i>-1$ (resp.~$a_i\geq -1$) for every $i$. 
\end{notation}

\begin{defn}[Relative normal crossing divisors]
Let $f:X \to Y$ be a smooth surjective morphism 
between smooth varieties with connected fibers and 
$D=\sum_i D_i$ a reduced divisor on $X$ such that 
$D^{h}=D$, where $D_i$ is a prime divisor for every $i$.  
We say that $D$ is {\em{relatively normal crossing}} 
if $D$ satisfies the condition that for each 
closed point $x \in X$, there exits 
an analytic open neighborhood $U$ and $u_1, 
\dots , u_k \in \mathcal O_{X,x}$ inducing a 
regular system of parameter on 
$f^{-1}f(x)$ at $x$, where $k=\dim f^{-1}f(x)$, such 
that $D\cap U =\{ u_1 \cdots u_l=0 \}$ 
for some $l$ with $0\leq l\leq k$. 
\end{defn}

Let us recall Kawamata's positivity theorem in \cite{kawamata}. 
It is the main ingredient of this paper. 

\begin{thm}[Kawamata's positivity theorem]\label{ka-th}
Let $f: X \to Y$ be a surjective morphism 
of smooth projective varieties with 
connected fibers. 
Let $P=\sum_{j} P_j$ and $Q=\sum_{l} Q_l$ be simple 
normal 
crossing divisors on $X$ and $Y$, respectively, 
such that $f^{-1}(Q)\subseteq P$ and $f$ is 
smooth over $Y \setminus Q$. 
Let $D=\sum_{j} d_jP_j$ be 
a $\mathbb Q$-divisor {\em{(}}$d_j$'s may 
be negative or zero{\em{)}}, which satisfies 
the following conditions{\em{:}}
\begin{itemize}
\item[(1)] $f:\Supp D^{h} \to Y$ is relatively 
normal crossing over $Y\setminus Q$ 
and $f(\Supp D^{v}) \subseteq Q$,
\item[(2)] $d_j<1$ unless $\codim _Yf(P_j)\geq 2$, 
\item[(3)] $\dim_{\mathbb C(\eta)}f_*\mathcal O
(\ulcorner 
-D\urcorner)\otimes_{\mathcal O_Y}\mathbb C(\eta)=1$, 
where $\eta$ is the generic point of $Y$, and
\item[(4)] $K_X+D \sim_{\mathbb Q} f^*(K_Y+L)$ 
for some $\mathbb Q$-divisor $L$ on $Y$. 
\end{itemize}
Let 
\begin{eqnarray} f^*(Q_l) &=& 
\sum_{j}w_{lj}P_j, \ \text{where}\ w_{lj}>0, \nonumber\\
\bar{d}_j&=&\frac{d_j+w_{lj}-1}{w_{lj}}\ 
\text{if}\ f(P_j)=Q_l, \nonumber\\
\delta_{l}&=& \mathrm{max}\{\bar{d_j}| f(P_j)=Q_l\}, \nonumber\\
\Delta_{0}&=&\sum \delta_{l}Q_{l},\ \text{and} \nonumber\\
M&=&L-\Delta_{0}.\nonumber
\end{eqnarray}
Then $M$ is nef. 
We sometimes call $M$ {\em{(}}resp.~$\Delta_0${\em{)}} 
the {\em{moduli part}} {\em{(}}resp.~{\em{discriminant 
part}}{\em{)}}.  
\end{thm}

\begin{rem}
In Theorem \ref{ka-th}, we note that $\delta_l$ can be 
characterized as follows. 
If we put 
$$
c_l=\sup \{t\in \mathbb Q\, | \, K_X+D+tf^*Q_l \ \text{is lc 
over the generic point of} \ Q_l\},  
$$ 
then $\delta_l=1-c_l$. 
\end{rem}

We give a remark on the Stein factorization. 
We will use Lemma \ref{04} in Section \ref{sec4}. 
See also Remark \ref{rem-stein} below. 

\begin{lem}[Stein factorization]\label{04}
Let $f:X\to Y$ be a smooth projective 
morphism between smooth varieties. 
Let 
$$f:X\overset{h}{\longrightarrow}
Z\overset{g}{\longrightarrow}Y$$
be the Stein factorization. 
Then $g:Z\to Y$ is {}\'etale. 
Therefore, $h:X\to Z$ is a smooth 
projective morphism between smooth varieties 
with connected fibers. 
\end{lem}
\begin{proof}
By assumption, $R^if_*\mathcal O_X$ is locally 
free and $$R^if_*\mathcal O_X\otimes \mathbb C(y)\simeq 
H^i(X_y, \mathcal O_{X_y})$$ for every $i$ and any $y\in Y$. 
By definition, $Z=\Spec _Yf_*\mathcal O_X$. 
Since $g_*\mathcal O_Z\simeq f_*\mathcal O_X$ is locally free, 
$g$ is flat. 
By construction, 
$$Z_y=\Spec H^0(X_y, \mathcal O_{X_y})$$ consists of 
$n$ copies 
of $\Spec \mathbb C$ for any $y\in Y$, 
where $n$ is the rank of $f_*\mathcal O_X$. 
Therefore, $g$ is unramified. 
This implies that $g$ is {}\'etale. 
Thus, $Z$ is a smooth variety and $h:X\to Z$ is a smooth 
morphism with connected fibers. 
\end{proof}

\section{Log Fano varieties}\label{sec3}

The proof of the following theorem is essentially the 
same as \cite[Theorem 1.2]{fujinoa}. 
We will use similar arguments in Section \ref{sec4}. 

\begin{thm}\label{06} 
Let $f:X\to Y$ be a proper surjective morphism 
between normal projective varieties 
with connected fibers. 
Let $\Delta$ be an effective $\mathbb Q$-divisor 
on $X$ such that 
$(X, \Delta)$ is klt. 
Assume that 
$-(K_X+\Delta+\varepsilon f^*H)$ is semi-ample, 
where $\varepsilon$ is a positive rational 
number and $H$ is an ample 
Cartier divisor 
on $Y$. 
Then we can find an effective 
$\mathbb Q$-divisor $\Delta_Y$ on $Y$ such that 
$(Y, \Delta_Y)$ is klt and 
$-(K_Y+\Delta_Y)$ is ample. 
In particular, if $K_Y$ is $\mathbb Q$-Cartier, then $-K_Y$ is big. 
\end{thm}

\begin{proof}
By replacing $H$ with $mH$ and $\varepsilon$ with 
$\frac{\varepsilon}{m}$ for some sufficiently 
large positive integer $m$, 
we can assume that $H$ is very ample and $\varepsilon <1$. 
By replacing $H$ with 
a general member of $|H|$, 
we can further assume that $(X, \Delta+\varepsilon f^*H)$ is 
klt. 
Let $A$ be a general member of a free linear system 
$|-m(K_X+\Delta+\varepsilon f^*H)|$ 
such that $(X,\Delta+\varepsilon f^*H+\frac{1}{m}A)$ 
is klt and
$$K_X+\Delta+\varepsilon f^*H+\frac{1}{m}A\sim_{\mathbb Q}0.$$
We put $\Gamma=\Delta+\varepsilon f^*H+\frac{1}{m}A$.
Then we consider the following commutative diagram: 
$$
\xymatrix{
   X' \ar[r]^{\nu} \ar[d]_{f'} & X \ar[d]^{f} \\
   Y' \ar[r]_{\mu} & Y,
} 
$$
where 
\begin{itemize}
\item[(i)] $X'$ and $Y'$ are smooth projective varieties,
\item[(ii)] $\nu$ and $\mu$ are projective birational morphisms,
\item[(iii)] we put $L=-K_{Y'}$ and 
define a $\mathbb Q$-divisor $D$ on $X'$ as follows: 
$$K_{X'}+D=\nu^*(K_X+\Gamma), $$ 
and 
\item[(iv)] there are simple normal 
crossing divisors $P$ on $X'$ and $Q$ on $Y'$ 
which satisfy the conditions (1) of 
Theorem \ref{ka-th} and there exists a 
set of sufficiently small non-negative 
rational numbers $\{s_l\}$ such 
that $\mu^*H-\sum_{l}s_lQ_l$ is ample. 
\end{itemize}
We see that $f':X'\to Y'$, $D$, and $L$ satisfy 
the conditions (1), (2), and (4) in 
Theorem \ref{ka-th}. Now we check the condition (3) in 
Theorem \ref{ka-th}. We put $h=f \circ \nu$.
\begin{claim}\label{clA}
$\mathcal O_Y = h_*\mathcal O_{X'}(\ulcorner  -D\urcorner)$
\end{claim}
\begin{proof}[Proof of {\em{Claim \ref{clA}}}] 
Since $(X, \Gamma)$ is klt, 
we see that 
$\ulcorner  -D\urcorner $ is effective and $\nu$-exceptional. 
Thus it holds 
that $\nu_{*}\mathcal O_{X'}(\ulcorner 
-D\urcorner)\simeq \mathcal O_{X}$. 
Since $f_{*}\mathcal O_{X}=\mathcal O_Y$, 
we have $\mathcal O_Y = h_*\mathcal O_{X'}
(\ulcorner  -D\urcorner)$.
\end{proof}
By Claim \ref{clA}, we 
see that $f':X'\to Y'$ and $D$ satisfy 
the condition (3) in Theorem \ref{ka-th} 
since $\mu$ is birational.
If we take $\mathbb Q$-divisors $\Delta_0$ and $M$ on $Y'$ as in 
Theorem \ref{ka-th}, then 
$$
K_{X'}+D\sim_{\mathbb Q}f'^*(K_{Y'}+M+\Delta_0) 
$$ and 
$M$ is nef. 
We have the following claim about $\Delta_0$.
\begin{claim}\label{clC}
$\Delta_0^{+}\geq\varepsilon \mu^*H$.
\end{claim}
\begin{proof}[Proof of {\em{Claim \ref{clC}}}] Since $H$ is 
general, $h^*H$ is reduced. We set $h^*H=\sum_{j} P_{k_j}$. 
Note that the coefficient of 
$P_{k_j}$ in $D$ is $\varepsilon$ for every $j$ 
by the generality of $H$ and $A$.
By the definition of $\bar{d}_{k_j}$, it holds that  
$$\bar{d}_{k_j}=d_{k_j}=\varepsilon.
$$
Thus we have $\Delta_0^{+}\geq \varepsilon \mu^*H$.
\end{proof}
We decompose $\varepsilon=
\varepsilon'+\varepsilon''$ such that 
$\varepsilon'$ and $\varepsilon''$ are positive 
rational numbers. Since $M$ is nef, 
$M+\varepsilon' (\mu^*H-\sum_{l}s_lQ_l)$ 
is ample. Hence, there exists an effective 
$\mathbb Q$-divisor 
$B$ such that 
$M+\varepsilon'(\mu^*H-\sum_{l}s_lQ_l) \sim_{\mathbb Q} B$, 
$(X, B+\varepsilon'\sum_{l}s_lQ_l+\Delta_0^{+}
+\varepsilon''\mu^*H)$ is klt, and 
$\Supp (B+\varepsilon'\sum_{l}s_lQ_l+\Delta_0^{+}
+\varepsilon''\mu^*H -\Delta_0^-)$ is 
simple normal crossing. If $\varepsilon'$ is a 
sufficiently small positive rational number, 
then we see that $$\Supp(B+\varepsilon'\sum_{l}s_lQ_l+\Delta_0^+
+\varepsilon''\mu^*H -\Delta_0^-)^-=\Supp\ \Delta_0^-.$$ 
We set 
$$\Delta_0'=\Delta_0^+-\varepsilon 
\mu^*H\ \ \text{and}\ \ \Omega'=B+
\varepsilon'\sum_{l}s_lQ_l+
\Delta_0'+\varepsilon''\mu^*H-\Delta_0^-.
$$
It holds that 
$$
K_{Y'}+\Omega'\sim_{\mathbb Q}K_{Y'}+L
\sim _{\mathbb Q}0. 
$$
By the following claim, $\mu_{*}\Omega'$ is effective.
\begin{claim}[cf.~Claim (B) in \cite{fujinoa}]\label{clB}
$\mu_{*}\Delta_0^-=0$.
\end{claim} 
\begin{proof}[Proof of {\em{Claim \ref{clB}}}]
Let $\Delta_{0}^{-} = -\sum_{k} 
\delta_{l_{k}}Q_{l_{k}}$, where $\delta_{l_{k}}<0$. 
If there exists $k$ and $j$ such that 
$\ulcorner  -d_j \urcorner < w_{l_{k}j}$, 
it holds that $-d_j+1 \leq w_{l_{k}j}$ 
since $w_{l_{k}j}$ is an integer. 
Then we obtain $\delta_{l_{k}}\geq 0$. 
Thus, it holds that $\ulcorner  -d_j \urcorner 
\geq w_{l_{k}j}$ for all $k$ and $j$. Therefore we have 
$\ulcorner  -D \urcorner \geq f'^*Q_{l_{k}}$. 
Since $\mathcal O_{Y'}=f'_{*}\mathcal O_{X'}$, we see that 
$f'_{*} \mathcal 
O_{X'}(\ulcorner  -D \urcorner) \supseteq \mathcal O_{Y'}(Q_{l_{k}})$. 
By Claim \ref{clA}, $\mu_{*}Q_{l_{k}}=0$. 
We finish the proof of Claim \ref{clB}. 
\end{proof}
We put $\Omega=\mu_{*}\Omega'$. 
Then we see that $\Omega$ is effective 
by Claim \ref{clB},  
$$K_{Y'} + \Omega'=
\mu^*(K_Y+\Omega),\ K_Y+\Omega \sim_{\mathbb Q} 
0,\ \text{and}\ \Omega \geq \varepsilon''H.$$
Thus $(Y, \Delta_Y)$ is klt and $-(K_Y+\Delta_Y)
\sim _{\mathbb Q}\varepsilon''H$ 
is ample if we put $\Delta_Y=\Omega-\varepsilon''H\geq 0$.  
We finish the proof of Theorem \ref{06}.  
\end{proof}

\begin{rem}\label{07} 
Let $(X, B)$ be a projective klt pair. 
Then $-(K_X+B)$ is semi-ample if and only if 
$-(K_X+B)$ is nef and abundant 
by \cite[Theorem 1.1]{fujino-ka}. 
\end{rem}

The following corollary is obvious by Theorem \ref{06}. 

\begin{cor}[{cf.~\cite[Theorem 2.9]{ps}}]\label{075} 
Let $f:X\to Y$ be a proper surjective morphism 
between normal projective varieties 
with connected fibers. 
Let $\Delta$ be an effective $\mathbb Q$-divisor 
on $X$ such that 
$(X, \Delta)$ is klt and $-(K_X+\Delta)$ is ample. 
Then there is an effective 
$\mathbb Q$-divisor $\Delta_Y$ on $Y$ such that 
$(Y, \Delta_Y)$ is klt and 
$-(K_Y+\Delta_Y)$ is ample. 
\end{cor}

We close this section with an easy corollary of Theorem \ref{06}. 

\begin{cor}\label{08} 
Let $(X, \Delta)$ be a projective 
klt pair such that $-(K_X+\Delta)$ is semi-ample. 
Let $n$ be a positive integer such that 
$n(K_X+\Delta)$ is Cartier. 
Then there is an effective $\mathbb Q$-divisor 
$\Delta_Y$ on 
$$
Y=\Proj \bigoplus _{m\geq 0}
H^0(X, \mathcal O_X(-mn(K_X+\Delta)))
$$ 
such that $(Y, \Delta_Y)$ is klt and 
$-(K_Y+\Delta_Y)$ is ample. 
\end{cor}

\begin{proof}
By definition, $Y$ is a normal projective variety and 
there is a projective surjective 
morphism $f:X\to Y$ 
with connected fibers such that 
$-(K_X+\Delta)\sim _{\mathbb Q}f^*H$, where 
$H$ is an ample $\mathbb Q$-Cartier 
$\mathbb Q$-divisor on $Y$. 
Then we can apply Theorem \ref{06}. 
\end{proof}

\section{Fano and weak Fano manifolds}\label{sec4}

In this section, we apply Kawamata's positivity theorem 
to smooth projective morphisms between 
smooth projective varieties. 

We note that the statement of the 
following theorem is weaker than 
\cite[Corollary 3.15 (a)]{debarre}. 
However, the proof of Theorem \ref{01} 
has potential for further generalizations. 
We describe it in details. 

\begin{thm}[{cf.~\cite[Corollary 3.15 (a)]{debarre}}]\label{01}
Let $f:X\to Y$ be a smooth projective morphism 
between smooth projective varieties with connected 
fibers. 
If $-K_X$ is semi-ample, 
then $-K_Y$ is nef. 
\end{thm}
\begin{proof}
Let $C$ be an integral curve on $Y$. 
Let $A$ be a general member of the free linear system $|-mK_X|$.  
Then there is a non-empty Zariski 
open set $U$ of $Y$ such that 
$C\cap U\ne \emptyset$ and that $A$ is smooth 
over $U$. 
By construction, $K_X+\frac{1}{m}A\sim _\mathbb Q0$. 
Let $\mu:Y'\to Y$ be a resolution 
such that 
$\mu$ is an isomorphism over $U$ and $\mu^{-1}(Y\setminus U)$ 
is a simple normal crossing divisor on $Y'$. 
We consider 
the following commutative diagram. 
$$
\xymatrix{\widetilde X=X\times _YY'\ar[r]^{\ \ \  \ \ \ 
\varphi} \ar[d]_{\widetilde f}& X\ar[d]^{f}\\ 
Y' \ar[r]^{\mu}&Y
}
$$ 
We note that 
$\widetilde f: \widetilde X\to Y'$ is smooth. 
We write $K_{Y'}=\mu^*K_Y+E$. 
Then $\Supp E=\Exc (\mu)$, where $\Exc(\mu)$ is the exceptional locus 
of $\mu$,  and 
$E$ is effective. 
We put 
$$
K_{\widetilde X}+\widetilde D=\varphi^*(K_X+\frac{1}{m}A)
\sim _{\mathbb Q}0. 
$$
Then 
$$\widetilde D=-\widetilde f^*E+\varphi^*\frac{1}{m}A.
$$ 
Note that $K_{\widetilde X}=\varphi^*K_X+\widetilde f^*E$. 
We put $U'=\mu^{-1}(U)$. 
Then $\mu:U'\to U$ is an isomorphism. 
Let $\psi:X'\to \widetilde X$ be a resolution such that 
$\psi$ is an isomorphism over 
$\widetilde f^{-1}(U')$ 
and that $\Supp A'\cup \Supp f'^{-1}(Y'\setminus U')$ 
is a simple normal crossing 
divisor, where $A'$ is the strict transform of 
$A$ on $X'$ and $f'=\widetilde f\circ \psi:X'\to Y'$. 
We define 
$$
K_{X'}+D=\psi^*(K_{\widetilde X}+\widetilde D)
\sim _{\mathbb Q}0. 
$$ 
We can write 
$$
K_{X'}+D=f'^*(K_{Y'}+\Delta_0+M)
$$ 
as in Kawamata's positivity theorem (see Theorem \ref{ka-th}). 
We put $E=\sum _i e_i E_i$, where 
$E_i$ is a prime divisor 
for every $i$ and $E_i\ne E_j$ for 
$i\ne j$. 
The coefficient of $E_i$ in $\Delta_0$ is $1-c_i$, 
where 
$$
c_i=\sup \{t\in \mathbb Q \,|\, K_{X'}+D+tf'^*E_i \ {\text{is 
lc over the generic point of}} \ E_i\}. 
$$ 
By construction, 
$$
c_i=\sup \{t\in \mathbb Q \,|\, K_{\widetilde X}+
\widetilde D+t\widetilde {f}^*E_i \ {\text{is 
lc over the generic point of}} \ E_i\}. 
$$ 
Since 
$$
\widetilde D=-\widetilde f^*E+\varphi^*\frac{1}{m}A, 
$$ 
and $\varphi^*\frac{1}{m}A$ is effective, 
we can write $c_i=e_i+a_i$ for some 
$a_i\in \mathbb Q$ with $a_i\leq 1$. 
Thus, we have $1-c_i=1-e_i-a_i$. 
Therefore, the coefficient of $E_i$ in $E+\Delta_0$ is 
$$e_i+1-e_i-a_i=1-a_i\geq 0. $$
So, we can see that $E+\Delta_0$ is effective. 
Since $K_{Y'}+\Delta_0+M\sim _{\mathbb Q}0$ and 
$K_{Y'}=\mu^*K_Y+E$, 
we have 
$$
-\mu^*K_Y=-K_{Y'}+E\sim_{\mathbb Q} E+\Delta_0+M. 
$$ 
Let $C'$ be the strict transform of 
$C$ on $Y'$. 
Then 
\begin{align*}
C\cdot (-K_Y)&=C'\cdot (-\mu^*K_Y)\\ 
&=C'\cdot (E+\Delta_0 +M)\geq 0. 
\end{align*} 
It is because $M$ is nef 
and $\Supp (E+\Delta_0)\subset Y'\setminus U'$. 
Therefore, $-K_Y$ is nef. 
\end{proof}

We give a very important remark on Theorem \ref{01}. 

\begin{rem}[Semi-ampleness of $-K_Y$]\label{011} 
We use the same notation as in Theorem \ref{01} and 
its proof. 
It is conjectured that 
the {\em{moduli part}} $M$ is semi-ample 
(see, for example \cite[0.~Introduction]{ambro1}). 
Some very special cases of this conjecture were 
treated in \cite{fujino-nagoya} before \cite{ambro1}. 
Unfortunately, the results in \cite{fujino-nagoya} 
are useless for our purposes here. 
If this semi-ampleness conjecture 
is solved, then we will obtain 
that $-K_Y$ is semi-ample. 

Let $y\in Y$ be an arbitrary point. 
We can choose $A$ such that 
$y\in U$. Since 
$$
-\mu^*K_Y\sim _{\mathbb Q}M+E+\Delta_0, 
$$
$E+\Delta_0$ is effective, and $\Supp (E+\Delta_0)\subset 
Y'\setminus U'$, 
we can find a positive integer $m$ and an effective Cartier 
divisor $D$ on $Y$ such that 
$-mK_Y\sim D$ and that $y\not\in \Supp D$. It 
implies that $-K_Y$ is semi-ample. 

By \cite{kawamata1}, $M$ is semi-ample if $\dim Y=\dim X-1$. 
Therefore, $-K_Y$ is semi-ample when $\dim Y=\dim X-1$. 

In \cite[Theorem 3.3]{ambro2}, 
Ambro proved that $M$ is nef and 
abundant. 
So, if $Y$ is a surface, then we can check 
that $-K_Y$ is semi-ample 
as follows. 
If $\nu (Y', M)=\kappa (Y', M)=0$ or $1$, 
then $M$ is semi-ample. 
Therefore, we can apply the same argument as above. 
If $\nu (Y', M)=\kappa (Y', M)=2$, then 
$M$ is big. 
Since 
$$
-\mu^*K_Y\sim _{\mathbb Q}M+E+\Delta_0
$$
and $E+\Delta_0$ is effective, 
$-\mu^*K_Y$ is big. 
Therefore, $-K_Y$ is nef and big. 
In this case, $-K_Y$ is semi-ample by 
the Kawamata--Shokurov base point free theorem. 
Anyway, for an arbitrary point $y\in Y$, 
we can always find a positive integer $m$ and an effective 
Cartier divisor $D$ on $Y$ such that 
$-mK_Y\sim D$ and that $y\not \in \Supp D$. 
It means that $-K_Y$ is semi-ample. 

In the end, in Theorem \ref{01}, 
$-K_Y$ is semi-ample if $\dim Y\leq 2$. 
By combining the above results, we know 
that $-K_Y$ is semi-ample 
when $\dim X\leq 4$. 
We conjecture that $-K_Y$ is semi-ample 
if $-K_X$ is semi-ample without 
any assumptions on dimensions. 
\end{rem}

\begin{rem}\label{012} 
In Remark \ref{011}, 
we used Ambro's results in \cite{ambro1} 
and \cite{ambro2}. 
When we investigate the moduli part $M$ on 
$Y$ by the theory of 
variations of Hodge structures, we note the 
following construction. 
Let $\pi:V\to X$ be a cyclic cover associated 
to $m(K_X+\frac{1}{m}A)\sim 0$. 
In this case, $\pi$ is a finite cyclic cover which 
is ramified only 
along $\Supp A$. Since $\Supp A$ is relatively 
normal crossing over $U$, we can 
construct a simultaneous resolution 
$f\circ \pi:V\to Y$ and 
make the union of the exceptional locus 
and the inverse image of 
$\Supp A$ a simple normal crossing divisor 
and relatively normal crossing 
over $U$ by the canonical 
desingularization theorem. 
Therefore, the moduli part $M$ on $X$ 
behaves well under pull-backs. 
It is a very important remark. 
\end{rem}

The semi-ampleness of $-K_Y$ is not so 
obvious even when $-K_X\sim _{\mathbb Q}0$. 
The proof of the following theorem depends 
on some deep results on 
the theory of variations of Hodge 
structures (cf.~\cite{ambro2} and \cite{fujino-ka}). 

\begin{thm}\label{torsion}
Let $f:X\to Y$ be a smooth projective morphism between smooth 
projective varieties. 
Assume that $-K_X\sim _{\mathbb Q}0$. 
Then $-K_Y$ is semi-ample. 
\end{thm}
\begin{proof}
By the Stein factorization (cf.~Lemma \ref{04}), 
we can assume that 
$f$ has connected fibers. 
In this case, we can write 
$$
K_X\sim _{\mathbb Q}f^*(K_Y+M), 
$$ 
where $M$ is the moduli part. 
By \cite[Theorem 3.3]{ambro2}, we know that 
$M$ is nef and abundant. 
Therefore, $-K_Y$ is nef and abundant. 
This implies that $-K_Y$ is semi-ample by 
\cite[Theorem 1.1]{fujino-ka}. 
\end{proof}
 
The following theorem is one of the main 
results of this paper. We note that 
it was proved by Yasutake in a special case where 
$f:X\to Y$ is a $\mathbb P^n$-bundle (cf.~\cite{yasutake}). 

\begin{thm}[Weak Fano manifolds]\label{02}
Let $f:X\to Y$ be a smooth projective 
morphism between smooth projective 
varieties. 
If $X$ is a weak Fano manifold, 
then so is $Y$.  
\end{thm}
\begin{proof}
By taking the Stein factorization, we can assume that 
$f$ has connected fibers (cf.~Lemma \ref{04}). 
By Theorem \ref{01}, $-K_Y$ is nef since 
$-K_X$ is semi-ample by the Kawamata--Shokurov 
base point free theorem. By 
Kodaira's lemma, we can find an effective $\mathbb Q$-divisor 
$\Delta$ on $X$ such that $(X, \Delta)$ is klt and 
that $-(K_X+\Delta)$ is ample. 
By Theorem \ref{06}, we can find an 
effective $\mathbb Q$-divisor 
$\Delta_Y$ such that $-(K_Y+\Delta_Y)$ is ample. 
Therefore, $-K_Y$ is big. 
So, $-K_Y$ is nef and big. This means 
that $Y$ is a weak Fano manifold. 
\end{proof}

The following example is due to Hiroshi Sato. 

\begin{ex}[Sato]\label{sato}
Let $\Sigma$ be the fan in $\mathbb R^3$ whose rays are generated by 
$$
\begin{array}{llll}
x_1=(1, 0, 1), &x_2=(0, 1, 0), &x_3=(-1, 3, 0), &x_4=(0, -1, 0), \\
y_1=(0, 0, 1), &y_2=(0, 0, -1),
\end{array}
$$
and their maximal cones are 
$$
\begin{array}{llll} 
\langle x_1, x_2, y_1\rangle, \langle x_1, x_2, y_2\rangle, 
\langle x_2, x_3, y_1\rangle, \langle x_2, x_3, y_2\rangle,\\
\langle x_3, x_4, y_1\rangle, \langle x_3, x_4, y_2\rangle, 
\langle x_4, x_1, y_1\rangle, \langle x_4, x_1, y_2\rangle. 
\end{array}
$$
Let $\Delta$ be the fan obtained from $\Sigma$ by successive star subdivisions 
along the rays spanned by $$z_1=x_2+y_1=(0, 1, 1)$$ and 
$$z_2=x_2+z_1=2x_1+y_1=(0, 2, 1). $$
We see that $V=X(\Sigma)$, the toric 
threefold corresponding to the fan $\Sigma$ 
with respect to the lattice 
$\mathbb Z^3\subset \mathbb R^3$, 
is a $\mathbb P^1$-bundle 
over $Y=\mathbb P_{\mathbb P^1}(\mathcal O_{\mathbb P^1}
\oplus \mathcal O_{\mathbb P^1}(3))$. 
We note that the $\mathbb P^1$-bundle structure $V\to Y$ is induced 
by the projection $\mathbb Z^3\to \mathbb Z^2:~(x, y, z)\mapsto (x, y)$. 
The toric variety $X=X(\Delta)$ corresponding to the fan $\Delta$ was obtained 
by successive blow-ups from $V$. 
We can check that $X$ is a three-dimensional toric weak Fano manifold and that 
the induced morphism 
$f:X\to Y$ is a flat morphism onto $Y$ since every fiber of $f$ is one-dimensional.  
It is easy to see that 
$-K_Y$ is big but not nef.  
\end{ex}

Therefore, if $f$ is only flat, then 
$-K_Y$ is not always nef 
even when $X$ is a weak Fano manifold. 

Let us give a new proof of the well-known theorem by 
Koll\'ar, Miyaoka, and Mori (cf.~\cite{kmm}). 
We note that $Y$ is not always Fano if 
$f$ is only flat. There exists an example in \cite{wis}. 

\begin{thm}[{cf.~\cite[Corollary 2.9]{kmm}}]\label{03}
Let $f:X\to Y$ be a smooth projective 
morphism between smooth projective 
varieties. If $X$ is a Fano manifold, 
then so is $Y$.  
\end{thm}
\begin{proof}
By taking the Stein factorization, we can assume that 
$f$ has connected fibers (cf.~Lemma \ref{04}). 
By Theorem \ref{02}, $-K_Y$ is nef and big. 
Therefore, $-K_Y$ is semi-ample by the 
Kawamata--Shokurov base point free theorem. 
Thus, it is sufficient to see that 
$C\cdot (-K_Y)>0$ for every integral curve $C$ on $Y$. 
Let $C$ be an integral curve $C$ on $Y$. 
We take a general very ample divisor $H$ on $Y$. 
Let $\varepsilon$ be a small positive rational number. 
Then $K_X+\varepsilon f^*H$ is anti-ample. 
Let $A$ be a general member of the 
free linear system $|-m(K_X+\varepsilon f^*H)|$. 
We can assume that there is a non-empty 
Zariski open set $U$ of $Y$ such that 
$H$ is smooth on $U$, $\Supp (A+f^*H)$ 
is simple normal crossing 
on $f^{-1}(U)$, $\Supp A$ is smooth 
over $U$, and 
$C\cap H\cap U\ne \emptyset$. 
Apply the same arguments as in the 
proof of Theorem \ref{01} to 
$$K_X+\varepsilon f^*H+\frac{1}{m}A\sim _{\mathbb Q}0.$$ 
Then we obtain a projective birational 
morphism $\mu:Y'\to Y$ from a smooth projective variety 
$Y'$ such that 
$\mu$ is an isomorphism over $U$ and 
$\mathbb Q$-divisors $\Delta_0$ and $M$ on $Y'$ as 
before. 
By construction, $\Delta_0$ contains $\varepsilon H'$, where 
$H'$ is the strict transform of $H$ on $Y'$ 
(cf.~the proof of Theorem \ref{06}). 
Therefore, we have 
$$C\cdot (-K_Y)=C'\cdot (E+\Delta_0+M)>0$$ 
as in the proof of Theorem \ref{01}.  
Thus, $-K_Y$ is ample.  
\end{proof}

We can prove the following theorem by the same arguments. 
It is a generalization of Theorem \ref{03}. 

\begin{thm}\label{05} 
Let $f:X\to Y$ be a smooth 
projective morphism between smooth projective 
varieties. 
Let $H$ be an ample Cartier divisor 
on $Y$. 
Assume that $-(K_X+\varepsilon f^*H)$ is semi-ample for some 
positive rational number $\varepsilon$. 
Then $-K_Y$ is ample, that is, 
$Y$ is a Fano manifold.  
\end{thm}

\begin{proof}
By Lemma \ref{04}, we can assume 
that $f$ has connected fibers. 
By Theorem \ref{06}, we see that $-K_Y$ is big. 
By the proof of Theorem \ref{03}, we can see that 
$C\cdot (-K_Y)>0$ for every integral curve $C$ on $Y$. 
By the Kawamata--Shokurov base point 
free theorem, $-K_Y$ is semi-ample. 
Thus, $-K_Y$ is ample. 
\end{proof}

\section{Comments and Questions}\label{sec5}

In this section, we will work over 
an algebraically closed field $k$ of arbitrary characteristic. 
We denote the characteristic 
of $k$ by $\ch k$. 

\begin{say}
Let $f:X\to Y$ be a smooth projective morphism 
between smooth projective varieties defined 
over $k$. 

\begin{enumerate}
\item[(A)] If $-K_X$ is ample, that is, 
$X$ is Fano, then so is $-K_Y$. 
\end{enumerate}
It was obtained by Koll\'ar, Miyaoka, and Mori in \cite{kmm}. 
Their proof is an application of the deformation theory of 
morphisms from curves invented by Mori. 
It needs mod $p$ reduction arguments 
even when $\ch k=0$. 
In the case $\ch k=0$, 
we gave 
a Hodge theoretic proof 
without using mod $p$ reduction arguments in Theorem \ref{03}. 

\begin{enumerate}
\item[(N)] 
If $-K_X$ is nef, then 
so is $-K_Y$. 
\end{enumerate}
This result can be proved by the same 
method as in \cite{kmm} (cf.~\cite{miyaoka}, 
\cite{zhang}, and \cite[Corollary 
3.15 (a)]{debarre}). 
In the case $\ch k=0$, 
we do not know whether we can prove it 
without mod $p$ reduction arguments or not. 
 
\begin{enumerate}
\item[(NB)] 
If $-K_X$ is nef and big, that is, $X$ is weak Fano, 
then so is $-K_Y$ when $\ch k=0$. 
\end{enumerate}  
It was proved in Theorem \ref{02}. 
We do not know whether this statement holds true or not in the case $\ch k>0$. 
See also Section \ref{sec-app}:~Appendix. 

\begin{enumerate}
\item[(SA)] 
If $-K_X$ is semi-ample, is $-K_Y$ semi-ample? 
\end{enumerate}
We have only some partial answers to this question. 
For details, see Remark \ref{011} and Theorem \ref{torsion}. 
In the case $\ch k=0$, 
we note that $-K$ is semi-ample if and only if 
$-K$ is nef and abundant (see Remark \ref{07}). 

\begin{enumerate}
\item[(B)] If $-K_X$ is big, 
is $-K_Y$ big? 
\end{enumerate}
The following example gives a negative answer to this question. 

\begin{ex}
Let $E\subset \mathbb P^2$ be a smooth cubic curve. 
We consider $f:X=\mathbb P_E(\mathcal O_E\oplus 
\mathcal O_E(1))\to E=Y$. 
Then, we see that 
$-K_X$ is big. However, $-K_Y$ is not big since $E$ is 
a smooth elliptic curve. 
\end{ex}

Anyway, it seems to be difficult to 
construct nontrivial 
examples. 
It is because the smoothness of $f$ is a very strong condition. 
\end{say}

We close this section with a remark on Lemma \ref{04}. 
It may be indispensable when $k\ne \mathbb C$. 

\begin{rem}\label{rem-stein} 
Lemma \ref{04} 
holds true even when $k\ne \mathbb C$. 
We can check it as follows. 
By the proof of Lemma \ref{04}, it is sufficient to 
see that 
$f_*\mathcal O_X$ is locally free and 
$f_*\mathcal O_X\otimes k(y)\simeq H^0(X_y, \mathcal O_{X_y})$ for 
every closed point $y\in Y$. 
Without loss of generality, we can assume that 
$Y$ is affine. 
Let us check that the natural map 
$$
f_*\mathcal O_X\otimes k(y)\to H^0(X_y, \mathcal O_{X_y})
$$ 
is surjective for every $y\in Y$. 
We take an arbitrary closed point $y\in Y$. 
We can replace $Y$ with $\Spec \mathcal O_{Y, y}$. 
Let $m_y$ be the maximal ideal corresponding to $y\in Y$. 
We note that 
$f_*\mathcal O_X\otimes k(y)\simeq (f_*\mathcal O_X)^{\wedge}_{y}\otimes 
k(y)$, where $(f_*\mathcal O_X)^{\wedge}_y$ is 
the formal completion of $f_*\mathcal O_X$ 
at $y$. 
By the theorem on formal functions (cf.~\cite[Theorem 11.1]{hartshorne}), we have 
$$
(f_*\mathcal O_X)^{\wedge}_y
\simeq \underset{\longleftarrow}{\lim}H^0(X_n, \mathcal O_{X_n}), 
$$ 
where $X_n=X\times _Y\Spec \mathcal O_{Y, y}/m^n_y$. 
Therefore, we can see that 
$$
(f_*\mathcal O_X)^{\wedge}_y\otimes k(y)\to H^0(X_y, \mathcal O_{X_y})
$$ is surjective. 
It is because $H^0(X_{yi}, \mathcal O_{X_{yi}})=k$ for every $i$, 
where $X_y=\coprod_i X_{yi}$ is the irreducible decomposition of a 
smooth variety $X_y$. 
By the base change theorem (cf.~\cite[Theorem 12.11]{hartshorne}), 
we obtain the desired results. 
\end{rem}

\section{Appendix}\label{sec-app}

In this appendix, we give another proof of 
Theorem \ref{thm1} depending on mod $p$ reduction arguments. 
This proof is not related to Kawamata's positivity theorem. 

First let us recall various results without proofs for the reader's convenience. 

\begin{say}[Preliminary results] 
The following theorem was obtained by the same way as in \cite{kmm}. 

\begin{thm}[{\cite[Corollary 3.15 (a)]{debarre}}]\label{kmm1} Let 
$f:X \to Y$ be a smooth morphism of 
smooth projective varieties over an arbitrary algebraic closed field. 
If $-K_X$ is nef,  then so is $-K_Y$.
\end{thm}

In \cite{ss}, Schwede and Smith established the following results on log Fano varieties and 
global $F$-regular varieties. For various definitions and 
details, see \cite{ss} and \cite{karen}.  

\begin{thm}[{cf.~\cite[Theorem 1.1]{ss}}]\label{ss1} Let $X$ be a normal projective 
variety over an $F$-finite field of prime 
characteristic. 
Suppose that $X$ is globally $F$-regular. 
Then there exists an effective $\mathbb Q$-divisor $\Delta$ on $X$ such 
that $-(K_X+\Delta)$ is ample and that $(X, \Delta)$ is klt. 
\end{thm}

For the definition of {\em{klt in any characteristic}}, see \cite[Remark 4.2]{ss}. 

\begin{thm}[{cf.~\cite[Theorem 5.1]{ss}}]\label{ss3} Let $X$ be a normal projective 
variety defined over a filed of 
characteristic zero. 
Suppose that there exists an effective $\mathbb Q$-divisor $\Delta$ on $X$ 
such that $-(K_X+\Delta)$ is ample and that $(X, \Delta)$ is klt. Then $X$ has 
globally $F$-regular type.
\end{thm}

\begin{thm}[{cf.~\cite[Corollary 6.4]{ss}}]\label{ss2} 
Let $f:X \to Y$ be a projective morphism of normal projective 
varieties over an $F$-finite field of prime characteristic. 
Suppose that $f_{*}\mathcal{O}_{X}=\mathcal{O}_{Y}$. 
If $X$ is a globally $F$-regular variety, then so is $Y$.
\end{thm}

We can find the following lemma in \cite[Proposition 3.7 (a)]{l}.

\begin{lem}\label{liu}Let $C$ be a smooth projective curve over a field $k$, 
let $K$ be an extension field of $k$, 
and let $D$ be a Cartier divisor on $C$. 
Suppose that $\pi:C_K:=C\times_{k}K \to C$ is the natural projection. 
Then $\mathrm{deg}_{k} D= \mathrm{deg}_{K} \pi^*D$.  
\end{lem}
 
By the above lemma, we see the following lemma. 

\begin{lem}\label{change}Let $X$ be a projective variety over a field $k$, 
let $K$ be an extension field of $k$, 
and let $D$ be a Cartier divisor on $X$. 
Suppose that $\pi^*D$ is nef, where $\pi:X_K:=X\times_{k}K \to X$ is the projection. 
Then $D$ is nef.
\end{lem}
\begin{proof}
We take a morphism $f:C\to X$ from a smooth projective curve. We consider the following 
commutative diagram:
 \begin{equation*}\label{eq:line}
\xymatrix{  C_K \ar[d]_{f_K} \ar[r]^{\pi_C} \ar@{}[dr]|\circlearrowleft & C \ar[d]^{f}\\
 X_K \ar[d] \ar[r]^{\pi} \ar@{}[dr]|\circlearrowleft & X \ar[d]\\
 \mathrm{Spec} K  \ar[r] & \mathrm{Spec} k  }
\end{equation*}
where $C_K:=C\times_{k}K$. By the assumption, 
$\mathrm{deg}_{K}{\pi_C}^*(f^*D) \geq 0$. 
Hence $\mathrm{deg}_{k} f^*D \geq 0$ by Lemma \ref{liu}. Thus $D$ is nef.
\end{proof}
\end{say}

Let us start the proof of Theorem \ref{thm1}. 

\begin{proof}[Proof of {\em{Theorem \ref{thm1}}}]
First, we note that $-K_X$ is semi-ample by the Kawamata--Shokurov base point 
free theorem and that $-K_Y$ is nef by Theorem \ref{kmm1}. 
It is sufficient to show that $(-K_Y)^{\dim Y}>0$. 
By the Stein factorization, we can assume that $f$ has connected fibers. 
We can 
take a finitely generated $\mathbb{Z}$-algebra $A$, 
a non-empty affine open set $U \subseteq \mathrm{Spec} A$, and smooth 
morphisms $\varphi : \mathcal{X} \to U$ and $\psi : \mathcal{Y} \to U$ such that  
$$\xymatrix{  \mathcal{X} \ar[dr] \ar[rr]^{F}&  & \mathcal{Y} \ar[dl]\\
&U&}$$
and  $F \simeq f$ over the generic point of $U$ and that $-K_{\mathcal {X}}$ is semi-ample. 
We take a general closed point $\mathfrak{p} \in U$. 
Note that the residue field $k:=\kappa(\mathfrak{p})$ of $\mathfrak{p}$ has 
positive characteristic $p$. Let $f_{p}:X_p \to Y_p$ be the fiber of $F$ at $\mathfrak{p}$, 
and let $K$ be an algebraic closure of $k$. By Theorem \ref{ss3}, 
we may assume that $X_p$ is globally $F$-regular. 
Let $\overline{f_p}:\overline{X_p} \to \overline{Y_p}$ 
be the base change of $f_p$ by $\mathrm{Spec} K$, 
where $\overline{X_p}:=X_p \times_{k} K$ and $\overline{Y_p}:=Y_p \times_{k}K$. 
Since $-K_{\mathcal X}$ is semi-ample, we see that 
$-K_{\overline{X_{p}}}$ is semi-ample. 
In particular, $-K_{\overline {X_p}}$ is nef. 
Hence, we obtain that $-K_{\overline{Y_{p}}}$ is nef by Theorem \ref{kmm1}. 
By Lemma \ref{change}, $-K_{Y_p}$ is nef. 
By Theorem \ref{ss2}, $Y_{p}$ is globally $F$-regular. 
Hence $-K_{Y_p}$ is nef and big. Thus $(-K_{Y_p})^{\mathrm{dim}Y}>0$. 
Since $\psi$ is flat, $(-K_{Y})^{\mathrm{dim}Y}>0$. Therefore, 
$-K_Y$ is nef and big.
\end{proof}


\end{document}